 \documentclass[11pt]{amsart}
\usepackage{amsmath}
\usepackage{amsthm}
\usepackage{mathrsfs}
\usepackage{comment}
\usepackage[usenames, dvipsnames]{color}
\def\bel{\begin{equation}\label}
\def\eeq{\end{equation}}
\def\bel{\begin{equation}\label}
\def\eeq{\end{equation}}
\newtheorem{Definition}{Definition}[section]
\newtheorem{Theorem}{Theorem}[section]
\newtheorem{Remark}{Remark}[section]
\newtheorem{Lemma}{Lemma}[section]
\newtheorem{Proposition}{Proposition}[section]

\newtheorem{Example}{Example}[section]
\newtheorem{Claim}{Claim}[section]
 
\title{  Lyapunov-like  functions \\involving Lie brackets}
\def\fudge{\mathchoice{}{}{\mkern.5mu}{\mkern.8mu}}

\def\bbc#1#2{{\rm \mkern#2mu\vbar\mkern-#2mu#1}}
\def\bbb#1{{\rm I\mkern-3.5mu #1}}
\def\bba#1#2{{\rm #1\mkern-#2mu\fudge #1}}
\def\bb#1{{\count4=`#1 \advance\count4by-64 \ifcase\count4\or\bba A{11.5}\or
\bbb B\or\bbc C{5}\or\bbb D\or\bbb E\or\bbb F \or\bbc G{5}\or\bbb H\or
\bbb I\or\bbc J{3}\or\bbb K\or\bbb L \or\bbb M\or\bbb N\or\bbc O{5} \or
\bbb P\or\bbc Q{5}\or\bbb R\or\bbc S{4.2}\or\bba T{10.5}\or\bbc U{5}\or%
\bbb P\or\bbc Q{5}\or\bbb R\or\bba S{8}\or\bba T{10.5}\or\bbc U{5}\or
\bba V{12}\or\bba W{16.5}\or\bba \v{11}\or\bba Y{11.7}\or\bba Z{7.5}\fi}}

\def \cf{{\rm CLF}}

\def \F {{\mathcal F}}

\def \R{I\!\!R}
\def \C {\mathcal{ T}}

\def \N {{\bb N}}

\def \ol {\overline}
\def\nn{I\!\!N}

\def\rr{I\!\!R}
\def \v{{\bf v}}
\def\ds{\displaystyle}
\def \vsm{\vskip 0.3 truecm}
\def \vv{\vskip 0.5 truecm}

\begin{document}
\begin{abstract}
For a given closed target we embed the
 dissipative  relation that defines  a control Lyapunov function
in a more general  differential  inequality involving Hamiltonians  built from   iterated  Lie brackets. The solutions of the resulting extended relation, here called {\it  degree-$k$ control Lyapunov functions ($k\geq 1)$}, turn out to be  still sufficient for the system to be globally asymptotically controllable to the target.   Furthermore,  we work out some examples where  no standard (i.e., degree-$1$)  smooth control Lyapunov functions   exist while  a {\it $C^\infty$}  degree-$k$ control Lyapunov function  does exist, for some $k>1$. The extension is performed  under very weak regularity assumptions on the system, to the point that,  for instance, (set valued) Lie brackets of locally Lipschitz vector fields are considered as well.

\end{abstract}
\author[M. Motta]{Monica Motta}
\address{M. Motta, Dipartimento di Matematica,
Universit\`a di Padova\\ Via Trieste, 63, Padova  35121, Italy}
 \email{motta@math.unipd.it}

\author[F. Rampazzo]{Franco Rampazzo}
\address{F. Rampazzo, Dipartimento di Matematica,
Universit\`a di Padova\\ Via Trieste, 63, Padova  35121, Italy}
\email{rampazzo@math.unipd.it}

\thanks{ This research is partially supported by  the Gruppo Nazionale per l' Analisi Matematica, la Probabilit\`a e le loro Applicazioni (GNAMPA) of the Istituto Nazionale di Alta Matematica (INdAM), Italy;  and
by Padova University grant PRAT 2015 ``Control
of dynamics with reactive constraints''}


\maketitle

\section{Introduction}
 A {\it  control Lyapunov function} (shortly, \cf) for a control system
\bel{E'}
 \left\{\begin{array}{l} \dot y =f(y,a) \\ \, \\
  y(0)=x\in\rr^n\backslash\C
  \end{array}\right.
\eeq
--where the control parameter $a$ ranges over a compact set of   controls,  and the (closed) subset $\C\subset \rr^n$   is regarded as  a {\it target}--  is a positive  definite function  $U:\overline{\rr^n\backslash \C}\to\rr$  \newline such that,  at each point $x\in {\rr^n\backslash \C}$,     the dynamics $f(x,a)$
 points in a direction along which $U$ is strictly decreasing, for a suitable choice of $a\in A$.  A wide literature investigates  the links between the existence  of a \cf \, and some properties of the  system-target pair.   Standard regularity assumptions   include local semiconcavity of $U$ in the interior of the domain of $U$, which, in  particular, allows  defining the set  of { limiting gradients} \footnote{See Definition \ref{DCLF}. Under this hypothesis,  $D^*U(x)$ coincides with the {\it  limiting subdifferential $\partial_LU(x)$},  largely used in the literature on Lyapunov functions.}  $ D^*U(x)$ at each $x\in \rr^n\backslash \C$. Therefore, the monotonicity of $U$ along suitable  directions of  $f$ can be expressed  by  means of the {\it dissipative}  differential inequality \begin{equation}\label{derivata}
H(x,D^*U(x) )  <0  \qquad\forall x\in \rr^n\backslash \C,
\end{equation}
where
$$
H(x,p):=\inf_{a\in A} \Big\langle p\, , f(x,a) \Big\rangle.
$$
Relation \eqref{derivata}  has to be interpreted as the occurrence, at each $x$, of the inequality $H(x,p) <0$  for every  $p\in  D^*U(x)$.
Since  $U$ is assumed to be  (proper and)  positive definite, by choosing controls verifying  \eqref{derivata} one is  ideally  looking for   trajectories  that run closer and closer to the target. More precisely, one has:
\begin{Theorem}\label{T1}If there exists a \cf, system {\rm  \eqref{E'}} is {\rm GAC} to $\C$.
\end{Theorem}
As customary, GAC to $\C$ is acronym of  {\it globally asymptotically controllable to $\C$}  (see Definition \ref{GAC}), which means that   for any  initial point $x$ there  exists a system trajectory $y(\cdot)$,   $y(0)=x$,   approaching the target $\C$ (in possibly infinite time), uniformly  with respect to the distance ${\bf d}(x,\C)$.

Results like Theorem \ref{T1} --of which some {``inverse"} versions exist as well --
lie at the basis of various constructions dealing, in particular, with stabilizability (see e.g. \cite{S2}, \cite{Ri} and the references therein).
Nonsmoothness is crucial for control Lyapunov functions: though relation \eqref{derivata} is a  partial differential   inequality --so admitting many more solutions than the corresponding Hamilton-Jacobi  equation-- in general no smooth control Lyapunov functions exist.  A great deal of effective ideas  has been flourishing during the last four decades to deal with this unavoidable lack of regularity  (see e.g. \cite{CLSS}, \cite{MaRS},  the books \cite{CLSW}, \cite{BR} and the references therein).
 Nevertheless, the regularity issue is of obvious  interest from a numerical point of view. In addition, any  feedback stabilizing strategy   would likely benefit from smoothness properties  of   a \cf\, (or of  some suitable \cf's replacement)-- in particular, in reference  with  sensitivity to data errors.

As an  attempt to reduce the unavoidability of nonsmoothness, in the present paper we   replace  relation \eqref{derivata}  with a less demanding inequality which  involves Lie brackets \footnote{We remind that the {\it Lie bracket}    of two $C^1$ vector fields $X,Y$   is defined (on any coordinate chart)  as $[X,Y]:=DY\cdot X-DX\cdot Y$.}. Let us assume that the dynamics is driftless control affine, namely:
 \bel{odeL}
 \left\{\begin{array}{l} \dot y =\ds \sum_{i=1,\ldots,m}a_if_i(y) \\ \, \\
  y(0)=x\in\rr^n\backslash\C  ,
  \end{array}\right.
  \eeq
and let  $A:=\{\pm e_1,\ldots,\pm e_m\}$ \footnote{More general control systems can be  considered: see Remark \ref{remgen} and subsection \ref{ssgen}. }.
Assume  the vector fields  $f_1,\ldots,f_m$ are of class $C^{k-1}$ for some integer $k\geq 1$. We will  define \footnote{See Section \ref{nonsmooth} for an extension of the notion of $H^{(2)}$ when the vector fields are locally Lipschitz but not $smooth$.} the {\it  degree-$k$ Hamiltonian $H^{(k)}(x,p)$} by setting
\bel{newH} H^{(k)}(x, p) := \inf_{v\in  { \mathcal F^{(k)}(x)}} \big\langle p ,v \big\rangle\qquad\forall (x,p)\in (\rr^n\backslash \C)\times\rr^n,
\eeq
where   ${ \mathcal F}^{(k)}$ denotes the family of iterated  Lie brackets  of degree $\leq k$ of the  vector fields $f_1,\ldots,f_m$.
(Notice, in particular,  that $H^{(1)} =H$).

A function  $U:\overline{\rr^n\backslash \C}\to\rr$ will be called a  {\it
 degree-$k$  control Lyapunov function} --shortly,  degree-$k$  \cf--  if  (it is positive definite, proper, semiconcave on  domain's interior, and) it verifies inequality
\bel{new} H^{(k)}(x, D^*U(x)) <0 \qquad\forall x\in \rr^n\backslash \C.
\eeq
Observe that, because of  \bel{chain}H^{(k)}\leq H^{(k-1)}\dots\leq H^{(1) }, \eeq
 relation \eqref{new} is weaker  than \eqref{derivata}.

Still, in view of Theorem \ref{th1int} below, the inequality \eqref{new} is sufficient for the system to be GAC to $\C$, as stated in the following result:
\begin{Theorem}\label{th1int} Let   a    degree-$k$  \cf \, exist, for some positive integer $k$. Then   system  {\rm  \eqref{odeL}}  is {\rm GAC} to $\C$.\end{Theorem}

The  use of Lie brackets  as {\it higher order directions} is widespread in Control Theory, both within necessary   conditions for optimality   and  within sufficient  conditions for  various kinds of  controllability (see e.g. \cite{AgSa}, \cite{BP}, \cite{Co}, \cite{K}, \cite{S1},\cite{Su},\cite{FHT}). Furthermore, they are involved in boundary conditions ensuring uniqueness  for Hamilton-Jacobi equations, e.g. in relation with continuity properties of the corresponding value function (see e.g. \cite{BCD},\cite{So}).  However,  here Lie brackets are directly involved in the proposed   differential inequalities.

As for the regularity issue,  we wish to remark that  a degree-$k$ control Lyapunov function, $k> 1$,  may happen to be  more regular than a standard (i.e., degree-$1$)  control Lyapunov function.   It  may even occur  the case where $H^{({{k}})}(x,D^*U(x)) <0$ for some   $C^\infty$ function $U$,   while no smooth $U$ satisfies the standard inequality $H(x,D^*U(x)) <0$ (see Examples \ref{es1}-\ref{es4} below).
 Let us observe that the two reasons why a  (degree-$1$) control Lyapunov function may result discontinuous are  : i) the  {\it shape} of the target's boundary $\partial\C$; ii)   the  {\it shortage}  of   dynamics' directions.  While  there is nothing one can do to remedy  i),  the introduction  of Hamiltonians $H^{(k)}$ \,  ($k>1$), which are minima over larger sets of directions,  is a way  to reduce the effects of  ii).

The regularity hypotheses in the case of degree-2 control Lyapunov functions are relaxed in 
  Section \ref{nonsmooth} in order to include  Lipschitz continuous vector fields. Since the classical brackets $[f_i,f_j]$ may happen to be  not even  defined  at possibly infinitely many  points, we  make use of the  generalized, set-valued brackets defined in
\cite{RS1}. Accordingly, the Hamiltonian  $H^{(2)}$ is  computed as a  min-max value. Let us remark that the  degree-$2$ control Lyapunov function of  Example \ref{esL} is $C^\infty$ despite the fact the vector fields are  not even $C^1$.

\subsection{Preliminaries and notation}\label{prel}
For the reader convenience, some classical concepts, like  global asymptotic controllability  to a set $\C$, in short GAC to $\C$, and
a few technical definitions are here recalled.
\vsm

 Given an  integer $k\ge 1$ and an open subset $\Omega\subseteq\R^n$, we write  $C^k(\Omega)$  to  denote the set of  vector fields of class $C^k$ on $\Omega$, namely,  $C^k(\Omega):= C^k(\Omega,\R^n)$.  
The subset  $C_b^k(\Omega)\subset C^k(\Omega)$  of functions with bounded derivatives (up to the order $k$) will be endowed with the norm
$$
\| f\|_k:=\sum_{i=0,\dots, k}\, \sup_{x \in \Omega} |f^{(i)}(x)| \qquad (f^{(0)}:=f)
$$
(which makes it a Banach space). Similarly,    $C^{k-1,1}(\Omega) \subset C^{k-1}(\Omega)$  denotes the  subset of  vector fields  whose $k-1$-th derivative is  locally Lipschitz continuous and  $C_b^{k-1,1}(\Omega)$ is the subset of $C_b^{k-1}(\Omega)$ with (globally) Lipschitz continuous $k-1$-th derivative.

\begin{Definition}\label{adm}  Let $k\geq 1$  be an integer, and let   $f_1,\dots,f_m$ be vector fields belonging to $C^{k-1}(\R^n\setminus\C)$.  For any initial condition $x\in \R^n\setminus\C$ and any measurable control $\alpha:[0,+\infty)\to A$, a  trajectory-control  pair $(y ,\alpha)(\cdot)$ will be called {\rm admissible}  if there exists   $T\le+\infty$ such that  $y(\cdot)$ is a solution of \eqref{odeL} defined on $[0,T)$ and
$$
\lim_{t\to T}{\bf d}(y(t))=0,
$$
where ${\bf d}(\cdot):= {\bf d}(\cdot,\C)$.  When $k>1$, we will use $y_x(\cdot,\alpha)$  to denote the unique (possibly local) forward  solution to the Cauchy problem  \eqref{odeL}.
 \end{Definition}

\begin{Remark}\label{Crem} {\rm The main  object of the paper consists in establishing relations involving Lie brackets, so that a certain regularity is necessary when $k>1$ (see also Section \ref{nonsmooth}). However, observe that  as soon as  $k=1$ the vector fields $f_1 , \ldots,f_m$  are  just continuous, so that solutions of the Cauchy problem \eqref{odeL} for a given control may be not unique.  }
\end{Remark}

To give the notion of global asymptotic  controllability,  we  recall that  ${\mathcal KL}$ is used to denote the set of  continuous functions  $\beta:[0,+\infty)\times[0,+\infty)\to[0,+\infty)$ such that: (1)\, $\beta(0,s)=0$ and $\beta(\cdot,s)$ is strictly increasing  and unbounded for each $s\ge0$; (2)\, $\beta(\delta,\cdot)$ is decreasing  for each $\delta\ge0$; (3)\, $\beta(\delta,s)\to0$ as $s\to+\infty$ for each $\delta\ge0$.

 \begin{Definition}\label{GAC} The control system in {\rm (\ref{odeL})} is   {\rm globally  asymptotically controllable to $\C$}  --shortly, {\rm (\ref{odeL})} is {\rm GAC to $\C$}--  provided  there is a  function $\beta\in{\mathcal KL}$
 such that,  for each initial state $x\in\R^n\setminus\C$,   there exists an admissible  trajectory-control pair $(y,\alpha)(\cdot)$  such that
\bel{bbound}
{\bf d}(y(t))\le \beta\big({\bf d}(x),t\big) \qquad \forall t\in[0,+\infty). \,\,
\footnote{ By convention, we fix an arbitrary $\bar x\in\partial\C$ and  formally establish that, if $T<+\infty$, the trajectory $y(\cdot)$ is prolonged to $[0,+\infty)$,
by setting $y(t)=\bar x$ for all $t\geq T$.}
\eeq
\end{Definition}

 Let us recall that if $g_1, g_2$ are $C^1$ vector fields on a differential manifold (of class $C^2$), their {\it Lie bracket} $[g_1,g_2]$ is the (continuous) vector field which  is  defined (on coordinate charts) by
$$
[g_1,g_2]= Dg_2\cdot g_1 - Dg_2\cdot g_1.
$$
Since $[g_1,g_2]$ turns out to be a vector field, provided sufficient regularity is assumed,  one can iterate the bracketing process so obtaining {\it iterated
Lie brackets}. We call {\em degree } of a given iterated bracket $B$  the number of objects appearing in $B$ (regarded as a formal object) when commas and left and right brackets are deleted. For instance, the degrees of $[[g_2,g_3],g_2]$,  $[[g_2,g_3],[g_2,g_4]]$, and $[g_4,[g_4,[g_4,[g_4,g_6]]]]$ are $3,4$, and $5$, respectively.

\vsm
Let us summarize  some basic notions in nonsmooth analysis  (see  e.g. \cite{CS}, \cite{CLSW} for a thorough treatment).
\begin{Definition}[Positive definite and proper functions]  A continuous function $F:\overline{\R^n\setminus\C} \to\R$ is said {\rm positive definite on $\R^n\setminus\C$} if  $F(x)>0$ \,$\forall x\in\R^n\setminus\C$ and $F(x)=0$ \,$\forall x\in\partial\C$. The function $F$ is called {\rm proper  on $\R^n\setminus\C$}  if the pre-image $F^{-1}(K)$ of any compact set $K\subset[0,+\infty[$ is compact.
\end{Definition}

 \begin{Definition} {\rm (Semiconcavity). }\label{sconc} Let $\Omega\subset\R^n$. A  continuous function $F:\Omega\to\R$  is said to be {\rm  semiconcave  on $\Omega$} if   there exist $\rho>0$ such that
 $$
 F(z_1)+F(z_2)-2F\left(\frac{z_1+z_2}{2}\right)\le \rho|z_1-z_2|^2,
 $$
 for all   $z_1$, $z_2\in \Omega$ such that $[z_1,z_2]\subset\Omega$. The constant $\rho$ above is called a {\rm semiconcavity constant} for $F$ in $\Omega$.  $F$ is said to be {\rm  locally semiconcave  on $\Omega$} if it semiconcave on every compact subset of $\Omega$.
 \end{Definition}
Let us remind that locally semiconcave functions are locally Lipschitz. Actually, they are twice differentiable almost everywhere.
 \begin{Definition}\label{D*}{\rm (Limiting gradient). }  Let $\Omega\subset\R^n$ be an open set, and let  $F:\Omega\to\R$  be a locally Lipschitz function.  For every $x\in \Omega$ let us set
$$
D^*{F}(x) \doteq \Big\{ w\in\R^n: \ \  w=\lim_{k}\nabla {F}(x_k), \ \  x_k\in DIFF(F)\setminus\{x\}, \ \ \lim_k x_k=x\Big\}
$$
where   $\nabla$ denotes the classical gradient operator and $DIFF(F)$ is the set of differentiability points of $F$.  $D^*{F}(x)$ is called
the  {\rm set of limiting gradients} of $F$ at $x$.
\end{Definition}
The set-valued map $x\mapsto D^*F(x)$ is upper semicontinuous on $\Theta$,  with nonempty, compact values. Notice that   $D^*{F}(x)$ is not convex.
  When $F$ is a locally semiconcave function,  $D^*{F}$ coincides  with the limiting subdifferential $\partial_LF$, namely,
  $$
  D^*F(x)=\partial_LF(x) := \{\lim \,  p_i: \  p_i\in \partial_PF(x_i), \ \lim\,  x_i=x\} \quad \forall x\in\Theta,
  $$
where  $\partial_PF$ denotes the proximal subdifferential,  largely used in the literature on Lyapunov functions.


\section{Degree-$k$ control Lyapunov functions}\label{S2}
\subsection{The main result}
 Let $k\geq 1$  be an integer.  Throughout the whole paper we assume that  the target  $\C\subset\R^n$  is a closed set with compact boundary and that $f_1,\dots,f_m$ are vector fields belonging to $C_b^{k-1}(\Omega\setminus\C)$  for any open, bounded subset $\Omega\subset\R^n$   (see Subsection \ref{prel}).

\begin{Definition}
Let us consider the family of vector fields
$$
\F^{(1)}:=  \left\{  f=\sum_{i=1}^ma_if_i, \quad a\in A\right\} = \big\{\pm f_i,\quad i=1,\ldots,m\big\}.
$$
Moreover,  if $k>1$,  for  every positive  integer $h$ such that  $2\leq h\le {{k}}$, set
$$\F^{(h)}:=\big\{B,\quad B \,\, \hbox{is a iterated Lie  bracket  of degree} \leq h\hbox{\,\,\, of}\,\,  f_1,\ldots,f_m\big\}$$
\end{Definition}
Clearly, every element of $\F^{(h)}$ is a vector field belonging to $C_b^{{k-h}}(\Omega\setminus\C)$  for any open, bounded subset $\Omega\subset\R^n$.
Notice that
\bel{inclu}\F^{(1)}\subseteq \F^{(2)}\subseteq\dots\subseteq\F^{(k)}
.\eeq
For every $h=1,\ldots,k$, let us introduce the set-valued map
$$
\F^{(h)}(x) := \Big\{X(x),\quad X\in \F^{(h)}\Big\} \qquad \forall x\in \R^n\setminus\C.
$$

\begin{Definition} For any   integer $1\le h\le k$,   let us define the {\em degree-$h$  Hamiltonian} $H^{({{h}})}$ corresponding to the control system \eqref{odeL},  by setting
$$
H^{({{h}})}(x, p):= \inf_{v\in {\F^{({{h}})}(x)}}\big\langle p ,\, v\big\rangle\qquad \forall (x,p)\in (\R^n\setminus\C)\times \rr^n.
$$
\end{Definition}
Under the above hypotheses the Hamiltonians $ H^{({{h}})}$ are well defined and   continuous.
 As already mentioned in the Introduction,   the degree-$1$  Hamiltonian $H^{(1)}$ coincides with the standard Hamiltonian:
 $$
 H^{(1)}(x,p) = H(x,p):=\displaystyle  \inf_{a\in A}\Big\langle p ,\, \sum_{i=1}^m a_i f_i(x)\Big\rangle.
 $$
  Morever, by \eqref{inclu} one gets
\bel{hminor}
H^{({{k}})} \le H^{({{k-1}})}\leq\dots\le H^{({{1}})}.
\eeq

\vskip5truemm
 \begin{Definition}\label{DCLF}
 We call {\em degree-$k$  control Lyapunov function} --in short,  {\rm degree}-$k$ \cf--
any  continuous  function $U:\overline{\R^n\setminus\C}\to\R$ such that the restriction  to ${\R^n\setminus\C}$ is  locally  semiconcave, positive definite, proper, and verifies
\bel{c2h}
 H^{({{k}})}(x,D^*U(x))  < 0\quad\forall x\in\R^n\setminus\C,
\eeq
the latter inequality meaning $ H^{({{k}})}(x,p)  < 0$  for each $p\in D^*U(x)$.
\end{Definition}

In Theorem \ref{Th3.1gen} below we prove  that the existence of a degree-$k$ control Lyapunov function, $k>1$, is  sufficient for  the system to be globally asymptotically controllable  to $\C$ (GAC to $\C$, see Definition \ref{GAC}), as in the classical  case $k=1$.

\begin{Theorem}\label{Th3.1gen}
Let us assume that, for some integer $k\geq 1$,    a degree-$k$ control Lyapunov function exists.
Then   system {\rm (\ref{odeL})} is GAC to $\C$.
\end{Theorem}
We postpone  the proof of Theorem \ref{Th3.1gen}  to the next section and  make some  general remarks. Furthermore, we   give   some  examples where, in particular,  the distance function  is a (possibly smooth)   {\rm degree}-$k$  {\rm CLF} for some $k>1$
 and {\it is not }    a  {\rm degree}-$1$  {\rm CLF}.


 \subsection{Remarks and examples}

\begin{Remark}\label{wr} {\rm The regularity assumptions can be slightly weakened in some cases by observing that, in order that  certain degree-$k$ brackets ($k>3$)  to be defined, it is not necessary that the vector fields are $k-1$ times differentiable. For instance, the bracket  $[[f_1,f_2],[f_3,f_4]]$ is well defined as soon as the vector fields $f_1,f_2,f_3,f_4$ are two times differentiable.  }
\end{Remark}

\begin{Remark}\label{remgen}
{\rm By suitably rescaling time, one can easily generalize  Theorem \ref{Th3.1gen} to the case when
the control set $A$  contains a ball of $\rr^m$ with positive radius.
By means of  linear algebraic and    relaxation arguments one can also try to  extend the result  up to the point of  admitting  sets $A$ such that $0$ is contained in the interior of the convex hull $co (A)$.}

\begin{Remark} {\rm
It is   easy to adapt Theorem \ref{Th3.1gen} to the case when the state space is an open set $\Omega\subset\R^n$, $\Omega\supset\C$.  In fact,  the thesis keeps unchanged as soon as  one requires the  degree-$k$ CLF\, $U:\Omega\setminus\overset{\circ} \C\to\R$ to verify all the assumptions in  Definition \ref{DCLF} in $\Omega$,   plus the following one:
$$
\exists U_0\in(0,+\infty]: \ \
\lim_{x\to x_0, \  x\in\Omega}U(x)=U_0 \ \  \forall x_0\in\partial\Omega;  \quad
U(x)<U_0 \quad\forall x\in\Omega\setminus\overset{\circ} \C.
$$}
\end{Remark}

\end{Remark}
 \begin{Remark}\label{InvL}  {\rm
 While  the fact that $U$ is   a degree-$k$ control Lyapunov function  implies that $U$ is also a  degree-$\bar k$ control Lyapunov function for every ${{\bar k>  k}}$, the converse is in general   false (see Example \ref{es1}). On the other hand, coupling {Theorem} \ref{Th3.1gen}   with  an {\it inverse Lyapunov} result like in   \cite{S2}, \cite{Ri},   it is easy to verify  that the existence of a  degree-$k$ control Lyapunov function, $k>1$, implies the existence of a standard (i.e., degree-$1$) control Lyapunov function. }
 \end{Remark}

\begin{Remark}{\rm
As in the case of standard (i.e. degree-$1$) \cf, the notion of degree-$k$ \cf\, is intrinsic, for vector fields, their Lie brackets, and the set of limiting gradients $D^*U$ are chart independent.  In particular the results in Theorems \ref{Th3.1gen} and \ref{Th3.1genL} are fit to be extended to Riemannian manifolds (where, of course, the notion of distance should coincide with the considered Riemannian metric).
Incidentally,  let us notice that we can define the Hamiltonians $H^{(k)}$
in terms of {\it Poisson brackets} \footnote{If $H(x,p)$ and $K(x,p)$ are differentiable functions, the {\it Poisson bracket} $\{H,K\}$ is defined by
$$
\{H,K\}(x,p):= \ds\sum_{i=1}^k \left(\frac{\partial H}{\partial x_i}\frac{\partial K}{\partial p_i} - \frac{\partial H}{\partial p_i}\frac{\partial K}{\partial x_i} \right)(x,p).
$$}.
Indeed, setting, for every vector field $X$
$$
H_{X}(x,p) :=\langle p  , X(x) \rangle
$$
 one has
$$
H^{(1)}(x,p)  = \inf\Big\{-| H_{f_i}(x,p)|, \  i=1,\ldots,m\Big\},
$$
$$
H^{(2)}(x,p) = \inf\Big\{H^{(1)}(x,p) ,-|\{H_{f_i},H_{f_j}\}(x,p)|,\  i,j=1,\ldots,m\Big\},
$$
$$
\begin{array}{c}
H^{(3)}(x,p) =
 \inf\Big\{H^{(2)}(x,p) ,
-|\{H_{f_i},\{H_{f_j},H_{f_\ell}\}\}(x,p)|, \  i,j,\ell=1,\ldots,m \Big\},
\end{array}
$$
and similarly for higher degrees.
}
\end{Remark}

  \begin{Example}\label{es1}{\rm Consider the  so-called {\it nonholonomic integrator}
$$
\dot y =a_1\,f_1(y)  + a_2\,f_2(y)  \,,
$$
where $\ds f_1:=\frac{\partial}{\partial x_1} - x_2\frac{\partial}{\partial x_3}$, $\ds f_2:=\frac{\partial}{\partial x_2} + x_1\frac{\partial}{\partial x_3}$.

By $[f_1,f_2] =\ds 2\frac{\partial}{\partial x_3}$ we get
$$H^{(1)}(x,p) = - \max\left\{|p_1-p_3x_2|\, , \, |p_2+p_3x_1|\right\}
$$
and
$$
H^{(2)}(x,p) = - \max\left\{|p_1-p_3x_2|\, , \, |p_2+p_3x_1|\, , \, 2|p_3|\right\}.
$$
Let $\C$ be a compact  target
 and let $U(\cdot)$ coincide with the distance ${\bf d}(\cdot)$ from $\C$. If there exists a point $\bar x\in\left(\{0\}\times\{0\}\times\rr\right)\cap \left(\rr^3\backslash \C\right)$ such that  $D^*(U)(x)\cap  (\{0\}\times\{0\}\times\rr) \neq \emptyset$,   then
$H^{(1)}(\bar x, D^*{U}(\bar x))= 0$. Therefore, in this case  the distance function  $U$ fails  to be a  degree-$1$ \cf. For instance,  this is the case when   $\C=\{x,\, |\,  |x|\leq \rho\} $ for some  $\rho\geq 0$. Indeed, $$
H^{(1)}((0,0,x_3), D^*{U}((0,0,x_3)))= 0,$$
 for all  $|x_3|> \rho$. In fact, when $\rho = 0$  no    degree-$1$  CLF of class $C^1$   exist  (see \cite{BR} and  \cite{Ri}).

Yet,   {\it $U$ is a  {\rm degree}-$2$  {\rm CLF}},  for whichever compact target $\C$. Indeed,  for every $x\in \rr^3\backslash \C$ one has  $|p|=1$  for all  $p\in D^*U(x)$, which implies
$$
H^{(2)}(x,D^*{U}(x))\leq  \max_{|p|=1}\Big\{- \max\left\{|p_1-p_3x_2|\, , \, |p_2+p_3x_1|\, , \, 2|p_3|\right\}\Big\} < 0,
 $$
for all $x\in \rr^3\backslash \C$.
In the case when    $\C=\{x,\, |\,  |x|\leq \rho\} $ for some $\rho\ge0$, $U$ is   a $C^1$ --actually,   $C^\infty$--  degree-$2$  \cf \,(hence $D^*U(x)=\{\nabla U(x)\}$).  Furthermore, one has
\bel{gammaes1}
H^{(2)}(x, D^*{U}(x))\leq-\frac{2}{3} \qquad  \forall x\in \rr^3\backslash \C.
\eeq
}
\end{Example}
 \begin{Example}\label{es2}{\rm
Now let us consider the system
$$
\dot y =a_1\,f_1(y)  + a_2\,f_2(y) 
$$
where $\ds f_1:=\frac{\partial}{\partial x_2} + x_2^2\frac{\partial}{\partial x_3}$, $\ds f_2:=\frac{\partial}{\partial x_2} + x_1^2\frac{\partial}{\partial x_3}$.

\noindent Let us compute the brackets of degree less than or equal to 3:
$$
[f_1,f_2](x) =2(x_1-x_2)\frac{\partial}{\partial x_3}  , \quad \big[f_1,[f_1,f_2]\big](x) =- \big[f_2,[f_1,f_2]\big](x)=2\frac{\partial}{\partial x_3}.
$$
 Therefore,
$$
H^{(1)}(x,p) = - \max\left\{|p_1+p_3x^2_2|\, , \, |p_2+p_3x^2_1|\right\},
$$
$$
H^{(2)}(x,p) = - \max\left\{|p_1+p_3x^2_2|\, , \, |p_2+p_3x_1^2|\, , \, 2|p_3(x_1-x_2)|\right\},
$$
and
$$
H^{(3)}(x,p) = - \max\left\{|p_1+p_3x^2_2|\, , \, |p_2+p_3x^2_1|\, , \, 2|p_3(x_1-x_2)|\, , \, 2|p_3| \right\}.
$$
For simplicity let us consider only the target $\C=\{0\}$. Once again, the distance function $U(x):=|x|$ is not a degree-$1$ \cf, since  $H^{(1)}(x,D^*{U}(x))=0$
for all   $x\in \{0\}\times\{0\}\times (\R\setminus\{0\})$.  $U$ is not even a  degree-$2$  \cf, for $H^{(2)}(x,D^*{U}(x)) = 0$ for all   $x\in \{0\}\times\{0\}\times (\R\setminus\{0\})$. However,   $|\nabla U(x)|=1$ (and $D^*U(x)=\{\nabla U(x)\}$) so that
$$
H^{(3)}(x,D^*{U}(x))\le   \max_{|p|=1}\,H^{(3)}(x,p) <0,
$$
and the distance $U$  is a  ($C^\infty$) degree-$3$  \cf.
}
\end{Example}

\begin{Remark}{\rm The control  systems in  Examples \ref{es1} and \ref{es2} verify a   Lie algebra rank condition  at each point \footnote{A system verifies the Lie algebra rank condition at $x$ if  the iterated Lie brackets linearly span $\rr^n$.}. Hence, by Chow-Rashevsky's Theorem, they are  small time locally controllable  at every point $x$, that is, the interior of the  reachable set from $x$  at any time contains $x$.  Actually, with akin arguments it is not difficult to  prove  the following general fact:

{\it  Let $\C\subset\rr^n$ be any target with compact boundary.  If a  system verifies the Lie algebra rank condition at {  every point} by means of brackets of degree $\leq k$  the distance function ${\bf d}(\cdot)$ from $\C$ is a degree-$k$ CLF.}

Indeed, since  $|p|=1$ for every $p\in D^*{{\bf d}}(x)$ and every $x\in\rr^n\backslash \C$, the Lie algebra rank condition implies that  for every such $x$ and $p$ there must exist  $w\in{\mathcal F}^{k}(x)$
such that  $\langle p,w\rangle < 0$.}
\end{Remark}

 While in the previous examples the minimum time function is finite at each point,  this is not the case of the following example, where  no trajectories  issuing  from points  $(x_1,x_2,x_3)$ such that  $x_3\neq 0$ can  reach the target (in finite time). Notice incidentally that  the  Lie algebra rank condition is violated  at each  point belonging  to the plane $x_3=0$.

\begin{Example}\label{es4}{\rm Consider the  system
$$
\dot y =a_1\,f_1(y)  + a_2\,f_2(y) \,,
$$
where $$\ds f_1:=\frac{\partial}{\partial x_1} -x_2 \phi(x_3)\frac{\partial}{\partial x_3}\quad \ds f_2:=\frac{\partial}{\partial x_2} +x_1\phi(x_3) \frac{\partial}{\partial x_3},$$ $\phi:\rr\to[0,+\infty[$ being a $C^1$ function
such that $\phi(x_3)=0$ if and only if $x_3=0$.
By
$
[f_1,f_2](x) = 2\phi(x_3)\frac{\partial}{\partial x_3}
$
one gets
$$
H^{(1)}(x,p) = - \max\left\{|p_1-p_3x_2\phi(x_3)|\, , \, |p_2+p_3x_1\phi(x_3)|\right\}
$$
and
$$
H^{(2)}(x,p) = - \max\left\{|p_1-p_3x_2\phi(x_3)|\, , \, |p_2+p_3x_1\phi(x_3)|\, , \, 2\phi(x_3)|p_3|\right\}.
$$
Let us consider again  the target $\C=\{0\}$. Also in this case the  distance function $U(x):=|x|$ is not a degree-$1$ \cf, since  $$H^{(1)}(x,D^*{U}(x))=0
\quad \forall x\in \{0\}\times\{0\}\times (\R\setminus\{0\}),$$
and  $U$ is  a still  degree-$2$  \cf, since
$$H^{(2)}(x,D^*{U}(x)) <0
$$
for all   $x\in \{0\}\times\{0\}\times (\R\setminus\{0\})$.
}
\end{Example}

Of course the fact that the Lie algebra rank condition is verified almost everywhere --as in the previous examples-- is far from being necessary for a CLF of whatever degree to exist. In fact, a system that  fails to  be  small time locally controllable  on large areas of its domain might not have any $C^1$ degree-$1$ CLF    while admitting a smooth  degree-$k$ CLF, for some $k> 1$, as illustrated in the following example:

\begin{Example}\label{es3}{\rm
Let $\varphi$, $\psi:[0,+\infty)\to [0,1]$ be   $C^\infty$ maps such that, for any $q\in\N$,
$$
\begin{array}{l}
\varphi(r)= 1\ \text{ if $r\in[2q,2q+1]$,} \\  \varphi(r)=0 \ \text{ if  $r\in[2q+(5/4),2q+(7/4)]$};
\end{array}
$$
$$
\begin{array}{l}
\psi(r)= 1\ \text{ if $r\in[2q+(7/8),2q+(17/8)]$,  } \\
   \psi(r)=0 \ \text{ if  $r\in[2q+(1/4),2q+(3/4)]\cup[0,(1/4)]$. }
   \end{array}
$$
Let us consider the control system
$$
\dot y = a_1\,f_1(y)  +a_2\, f_2(y)+  a_3\, f_3(y)\,,
$$
where
$$\begin{array}{c}
\ds
f_1=\varphi(|x|)\left(\frac{\partial}{\partial x_1} -x_2\frac{\partial}{\partial x_3}\right)
, \quad f_2= \varphi(|x|) \left(\frac{\partial}{\partial x_2} +x_1\frac{\partial}{\partial x_3}\right),
\\\,\\\ds
f_3= \psi(|x|) \left(x_1\frac{\partial}{\partial x_1} +x_2\frac{\partial}{\partial x_2}+ x_3\frac{\partial}{\partial x_3}\right) .
\end{array}
$$
   Clearly the system is not  small time locally controllable at every point $x$ such that
 $2q+(5/4)\le|x|\le 2q+(7/4)$. Let the target $\C$ coincide with the origin $\{0\}$ and, again,  let us set $U(x):={\bf d}(x) =|x|$. For every $q\in\N$ one has
 $
 H^{(1)}(x,D^*{U}(x))=-|x| \quad\text{for all $x$ such that  $2q+(5/4)\le|x|\le 2q+(7/4)$.}
 $
 Furthermore, $H^{(1)}(x,D^*{U}(x))=0$ for   every $x$ such that  $x_1=x_2=0$ and $|x_3|\le 1$ or  $2q+(1/4)\le|x_3|\le 2q+(3/4)$, $q\ge 1$\footnote{Actually, there are no $C^1$ degree-$1$ \cf, as it can be proved  by noticing that the system coincides with the nonholonomic integrator in a whole neighborhood of the target.}. However, one easily checks that
 $$
H^{(2)}(x,D^*{U}(x))= -\frac{1}{|x|}\, \max\left\{|x_1-x_3x_2|\, , \, |x_2+x_3x_1|\, , \, 2|x_3|\right\} \le -\frac{2}{3}   ,
 $$
 for all  $x\in \R^3\setminus\C$, so that $U$ is a ($C^\infty$) degree-$2$ Lyapunov function.
}\end{Example}

\section{Proof of Theorem \ref{Th3.1gen}}\label{S3}

The case when $k=1$ has already been proved in \cite{MR}, where the hypotheses are even weaker than the ones assumed here (for instance, vector fields are allowed to be unbounded near the target). So we will always assume $k>1$: in particular, there will be a unique trajectory $y_x(\cdot, \alpha)$ corresponding to an initial condition $x$  and a control $\alpha(\cdot)$.
\subsection{Preliminary facts}
To begin with, let us  point out  that  the $0$ in   the dissipative  relation   can be replaced by a nonnegative function of $U$:

\begin{Proposition}\label{claim1}  Let  $U:\overline{\R^n\setminus\C} \to\R$   be a continuous function, such that $U$ is   locally  semiconcave, positive definite and proper on  ${\R^n\setminus\C}$.  Then the conditions (i) and (ii) below are equivalent:

\begin{itemize}\item[(i)]  $U$ verifies
\bel{c22} H^{({{2}})}(x, D^*U(x))<0\eeq
for all $x\in\R^n\setminus\C$;

\item[(ii)] for every $\sigma>0$  there exists a continuous, strictly increasing, function  $\gamma:[0,+\infty)\to:[0,+\infty)$ such that
\bel{c2'}
H^{({{2}})}(x,D^*U(x)) \le-\gamma(U(x))
\eeq
for all $x\in U^{-1}((0,2\sigma])$.
\end{itemize}
\end{Proposition}
Notice that  the only non trivial implication, namely (i)$\implies$(ii), is a simple consequence of the upper semicontinuity of the set-valued map $x\mapsto D^*U(x)$  on the compact sets $U^{-1}([u,2\sigma])$ ($u\in(0,2\sigma)$) and of the upper semicontinuity of  $H^{(k)}$. For a detailed proof, we refer  to \cite[Proposition 3.1]{MR}.

\begin{Remark}\label{gradienti}{\rm In a  good deal of  literature on control Lyapunov functions one utilizes   the   proximal subdifferential $\partial_PU(x)$   as a nonsmooth substitute for the derivative of $U$ (see \cite{CLSW}). However,  in view of Proposition \ref{claim1}, our  using  the set of  limiting  gradients  $D^*U(x)$  (which coincides with the limiting subdifferential $\partial_LU(x)$, see Subsection \ref{prel}) is equivalent to using  the  proximal subdifferential.  Indeed, for any $N>0$ condition  $H^{({{k}})}(x,D^*U(x))  \le-\gamma(Ux))$ is equivalent to $H^{({{k}})}(x,\partial_LU(x))  \le-\gamma(Ux))$ for any $x\in U^{-1}((0, N])$. On the other hand,   by the construction of   $\partial_LU(x)$ and by the  continuity of $H^{(k)}(\cdot)$,  this   holds true as soon as  $H^{({{k}})}(x,\partial_PU(x))  \le-\gamma(Ux))$.}
\end{Remark}

 Secondly,  basic properties of the semiconcave functions  imply  the following fact (see e.g. \cite{CS}):
\begin{Lemma}\label{Lscv} Let  $U:\overline{\R^n\setminus\C} \to\R$   be a continuous function, such that $U$ is   locally  semiconcave, positive definite, and proper on  ${\R^n\setminus\C}$.
Then for any compact set $\mathcal{K}\subset \R^n\setminus\C$ there exist some positive constants  $L$ and $\rho$ such that, for any $x\in \mathcal{K}$
 \footnote{The inequality (\ref{scv}) is usually formulated with the proximal superdifferential  $\partial^P F$. However, this does not make a difference here since $\partial^P F=\partial_C F=co D^* F$ as soon as $F$ is locally semiconcave. Hence (\ref{scv}) is true in particular for $D^*U$. },
\bel{scv}
\begin{array}{l}
|p|\le L    \quad \forall p\in D^*U(x),  \\ \, \\
U(\hat x)-U(x)\le  \langle p,\hat x-x \rangle+\rho|\hat x-x|^2,
\end{array}
\eeq
for any point  $\hat x\in\mathcal{K}$ such that $[x,\hat x]\subset  \mathcal{K}$.

Moreover, if $\mathcal{K}_1,\mathcal{K}_2\subset \R^n\setminus\C$ are compact subsets and  $\mathcal{K}_1\subseteq \mathcal{K}_2$, we can choose the corresponding constants   $L_1$,  $\rho_1$ and  $L_2$,  $\rho_2$
such that  $L_1\leq L_2$ and $\rho_1\leq \rho_2$.
\end{Lemma}

\subsection{ A degree-$k$  "feedback"}
Let us  introduce a notion  of  {\rm degree-$k$ feedback}.  For a given $\sigma>0$, let  ${{\gamma}}$  be a  function as in
Proposition \ref{claim1}, and let $x\mapsto p(x)$ be a selection of  $x\mapsto D^*U(x)$ on $U^{-1}((0,2\sigma])$, so that
$$
H^{({{k}})}(x,p(x)) \le-\gamma(U(x)) \qquad \forall x\in U^{-1}((0,2\sigma]).
$$

 \begin{Definition}\label{Dfed} For a given  $\sigma>0$, let  ${{\gamma}}(\cdot)$, and $p(\cdot)$ be chosen as above.
A selection  $$ \v :  U^{-1}((0,2\sigma]) \to \rr^n , \qquad
 x  \mapsto \v(x)\in \F^{({{k}})}(x)$$  is called  a {\em  degree-$k$ feedback}   (corresponding to $U$, $\sigma$, ${{\gamma}}(\cdot)$, and $p(\cdot)$)  if for every $x\in U^{-1}((0,2\sigma])$  there exists a positive integer $h\le k$ such that
\bel{feed1}
\left\{\begin{array}{l}
\v(x)\in\F^{(h)}(x), \\\,\\
 \Big\langle p(x) ,\, \v(x)  \Big\rangle \le -{{\gamma}}(U(x)),   \\ \,\\
 \text{and, if $h>1$:} \quad H^{({{h-1}})}(x,p(x))>-\gamma(U(x)).
\end{array}\right.
\eeq
The number  $h$ will be called the {\em degree of the feedback  $\v$ at  $x$}.
\end{Definition}

Let   us momentarily assume that there exists $\hat M\geq 0$ such that
\bel{boundfield}
\|f_i\|_{{k-1}} \leq \hat M\qquad \forall i=1,\ldots,m
\eeq
in the whole set $\R^n\setminus\C$, which, in view of the compactness of the control set $A,$  implies that there is $M\geq 0$ verifying
\bel{fieldbounded}
\|X\|_0\leq M
\eeq
for any iterated bracket  $X$ in  $\F^{(k)}$. Under this assumption
  one can regard each vector $\v(x)$ as a {\it tangent} vector to a curve that is a  suitable composition of  flows, as stated in the following result (see e.g. \cite{FR1}, \cite{FR2}):
\begin{Lemma}\
\label{applemma} Under assumption  \eqref{boundfield} there exists a real constant  $c>0$ such that for any $x\in \R^n\setminus\C$, any feedback    $\v(\cdot)$ of degree $h$ at $x$,  and any $t>0$,
 one can find a control
 $\alpha_t:[0,t]\to A$  such that
\begin{itemize}
\item[(i)] $\alpha_t(\cdot)$ is constant
on intervals $\displaystyle\left[\frac{jt}{r} ,\frac{ (j+1)t}{r}\right)$, $j=0,\ldots,r-1$;
\item[(ii)]  the estimate
\bel{yt}
\left|y_x(t,\alpha_t)-x-\frac{\v(x) }{r^{h}}\,t^h\right|\leq\frac{c}{r^{h}}\, t^{h+1}
\eeq
holds true, where $r$ is an integer depending on  the formal Lie bracket corresponding to $\v(x)$ and is increasing with the degree.
\end{itemize}
\end{Lemma}
For instance, $r=1,4,10$ if $h=1,2,3$, respectively. In particular,  if   $\v(x) = [[f_1,f_2], f_3](x)$  one sets
$$
\alpha_t(s) :=\left\{
\begin{array}{ll}
 & e_1 \quad\text{if $s\in \left[0,t/10\right)\cup \left[6t/10,7t/10\right)$} \\
 & e_2 \quad\text{if $s\in \left[t/10,2t/10\right)\cup \left[5t/10,6t/10\right)$}  \\
 & e_3 \quad\text{if $s\in \left[4t/10,5t/10\right)$}  \\
  -&e_1\quad\text{if $s\in \left[2t/10,3t/10\right)\cup \left[8t/10,9t/10\right) $}  \\
    -&e_2 \quad\text{if $s\in  \left[3t/10,4t/10\right)\cup\left[7t/10,8t/10\right)$}  \\
  -&e_3 \quad\text{if $s\in \left[9t/10,t\right)$.}
\end{array}\right.
$$
Let us point out that  one can have different $r$'s for feedbacks having the same degree \footnote{Precisely: For each formal bracket $B$, the corresponding  $r=r(B)$ is defined recursively: one sets $r(B)=1$ if $B$ has degree $1$, while, if $B=[B_1,B_2]$ and $r_1=r(B_1)$, $r_2=r(B_2)$ one sets $r(B):=2(r_1+r_2)$. For instance, $r([g_1,g_2])=4$, $r([g_1,[g_2,g_3]])=10$, $r([g_1,[g_2,[g_3,g_4]]])=22$ and $r([[g_1,g_2],[g_3,g_4]]=16$.}.

\subsection{ A step of degree $h\leq k$}

Now let us choose   $z\in U^{-1}((0,\sigma])$  and  a feedback   $\v$ of degree ${{k}}$ (surely existing by \eqref{c2h}). Let the feedback  $\v$ have  degree $h$ at $z$. We shall rely on the following result.

\begin{Claim}\label{claim2h}  Let $U(\cdot)$,  $\sigma$, $ \gamma(\cdot)$ and  $p(\cdot)$ as above. Furthermore let $\v(\cdot)$ be a degree-$k$ feedback corresponding to these data.
Then there exists a time-valued  function
$$
\tau:(0,\sigma]\times \{1,\dots, {{k}}\}\to (0,1]\,,
$$
 such that
\begin{itemize}
\item[i)] $j\mapsto \tau(u,j)^j$ and  $j\mapsto\tau(u,j)^{j-1}$ are decreasing  for every $u\in ]0,\sigma]$,
\item[ii)]  $u\mapsto \tau(u,j)$ is increasing  for every $j\in\{1,\dots,k\}$,  and
\item[iii)] for all  $z\in U^{-1}((0,\sigma])$ with a feedback    $\v(\cdot)$ of degree $h$ at $z$,  one has
 $$
U(y_z(t,\alpha_t))-U(z)\le  - \frac{\gamma(U(z ))}{2}  \left(\frac{t}{r}\right)^{h}\,\quad \forall t\in\left[0,\,\tau(U(z),h)\right]\,,
$$
where $r$ and   $y_z(\cdot,\alpha_t)$  are an integer and a  trajectory  associated to  $\v(z)$ as in Lemma \ref{applemma}.
\end{itemize}
\end{Claim}
 \begin{proof} Let $\nu>0$ be such that $U^{-1}((0,2\sigma])\subset B\left(\C,  \frac{\nu}{2} \right)$ and fix $z\in U^{-1}((0,\sigma])$ with a feedback    $\v(\cdot)$ of degree $h$ at $z$. To begin with we wish to choose a time $\bar\tau$ such that, for any $t\in[0,\bar\tau]$;
\begin{itemize}
\item[ i) ] $y(t)\in  B(\C,\nu)$  for any system's trajectory $y(\cdot)$ issuing from a point of $U^{-1}((0,\sigma])$;
\item[ii)]  $y_z^t(t)\in B\left(z, \frac{{\bf d}(z)}{2}\right)$ for any trajectory $y_z^t(\cdot):=y_z(\cdot,\alpha_t)$ associated to $\v(z)$   as in Lemma \ref{applemma}.
\end{itemize}
For this purpose, it is clearly sufficient to set
$$
\bar\tau(u,j):=\min\left\{\frac{\nu}{2M}, \sqrt[j]{ \frac{ {\bf d}(U^{-1}(u)) }{2M}}\right\} \quad \forall (u,j)\in(0,\sigma]\times\{1,\dots,k\},
$$
and to choose  $$\bar\tau:=\bar\tau(U(z),h).$$
Because of  ${\bf d}(U^{-1}(U(z))) \leq  {\bf d}(z)$ and  \eqref{yt},   the distance ${\bf d}([z,y_z^t(t)])$ between  the segment $[z,y_z^t(t)]$ and the target $\C$ verifies
$$
{\bf d}([z,y_z^t(t)])\geq   \frac{ {\bf d}(z) }{2}\ge\frac{ {\bf d}(U^{-1}(U(z))) }{2} ,
$$for every  $t\in[0, \bar\tau]$.
For every $u\in(0,\sigma]$,   in relation with the compact set
$$
{\mathcal K}(u):=\left\{x: \    \frac{{\bf d}(U^{-1}(u) )}{2}\le {\bf d}(x)\le\nu\right\},
$$
let   $L(u)$ and $\rho(u)$
be a Lipschitz continuity and a semiconcavity constant,
whose existence is stated in Lemma \ref{Lscv}. Let us set   $L:=L(U(z))$ and $\rho:=\rho(U(z))$.
By \eqref{yt},  for any $t\in[0,\bar\tau]$
we get
\bel{tau}
\begin{array}{c}
U(y^t_z(t))-U(z)\le  \Big\langle p(z),y^t_z(t)- z\Big\rangle +
\rho    \big|y^t_z(t)- z\big |^2   \le  \\\,\\  \displaystyle \Big\langle p(z), {\v(z)}\Big\rangle\left(\frac{t}{r}\right)^{h} +   |p(z)|ct   \left(\frac{t}{r}\right)^{h}
 +
\rho  \left(\frac{ t}{r} \right)^{2h}  (M +ct)^2\le
  \\\,\\\
\displaystyle
- \gamma(U(z))\,\left(\frac{t}{r}\right)^{h} +   |p(z)|ct   \left(\frac{t}{r}\right)^{h}
 +
\rho   \left(\frac{ t}{r} \right)^{2h}  (M +ct)^2\leq
\\\,\\
\displaystyle \left(
- \gamma(U(z))\, +    L ct +
\rho    \left(\frac{ t}{r} \right)^{h}(M+ct)^2\right)
\left(\frac{t}{r}\right)^{h}
\end{array}
\eeq
Let us observe that $\left(\frac{ t}{r} \right)^{h}\le t\le 1$ as soon as $t\le 1$.  Therefore,  if we define, for every $u$,
\bel{tauminore}
 \check\tau(u): = \frac{ \gamma(u)}{2[L(u)c+\rho(u) (M+c)^2]}  ,
\eeq
 and  we set
\bel{t2}
\tau(u,j):=\min\left\{ 1, \bar \tau(u,j), \check\tau(u)\right\} \quad \forall (u,j)\in(0,\sigma]\times\{1,\dots,k\},
\eeq
we get
$$
 L c t+\rho    \left(\frac{ t}{r} \right)^{h}(M+ct)^2\le  t[L c +\rho  (M+c)^2]\le  \frac{ \gamma(U(z))}{2}.
$$
as soon as $t\in[0,\tau]$, $\tau:=\tau(U(z),h)$.

Therefore, with this choice of $\tau$ we obtain
\bel{tau2}
U(y^t_z(t))-U(z)\le   - \frac{\gamma(U(z ))}{2}  \left(\frac{t}{r}\right)^{h} \quad \forall t\in [0,\tau].
\eeq
Moreover, $u\mapsto\tau(u,j)$ is increasing for every $j$: indeed,  by Lemma \ref{Lscv}, the   constants  $L(u)$ and $\rho(u)$  turn out to be decreasing   in $u$.
Finally, the fact that $j\mapsto \tau(u,j)^j$ and  $j\mapsto \tau(u,j)^{j-1}$  are decreasing is an easy consequence of the definition of $\tau(u,j)$ in \eqref{t2}.
The claim is now proved.
\end{proof}

 \subsection{Piecewise $C^1$ trajectories approaching the target}
 Now, let us define recursively a sequence of times $(t_j)_{j\ge0}$,
 of trajectory-control  pairs
 $(y_j(\cdot), \alpha_j(\cdot)):\left[s_{j-1},s_{j }\right]\to \R^n\times A$, $j\ge 1$, $s_0:=0$,
  $s_j:=s_{j-1} + t_i$, and points $x_j$
 as follows:
 \begin{itemize}
 \item
 $t_0:=s_0=0$,\, $x_1:=x$\,;
  \item
if $j\ge 1$, $t_j :=\tau(U(x_{j}),h_{j})$, where $h_{j}$ is the degree  of the feedback  $\v$ at $x_{j}$ and  $\tau(\cdot,\cdot)$ is as in  Claim \ref{claim2h};
  \item
  $(y_1,\alpha_1):[s_0,s_1]\to\R^n\times A$ is the trajectory-control pair
  defined as $(y_1,\alpha_1):=(y_{x_1}^{t_1}, \alpha_{t_1})$;
  \item for every  $j> 1$,
$
y_{j }(s_{j-1}):= y_{j-1}(s_{j-1}):= x_{j}\,,
$
 and the pair $(y_j(\cdot), \alpha_j(\cdot)):\left[s_{j-1},s_{j }\right]\to \R^n\times A$
 is given by $(y_j,\alpha_j):=(y_{x_j}^{t_j}, \alpha_{t_j})(s-s_{j-1})$ for every
$s\in\left[s_{j-1},s_{j }\right]$.
 \end{itemize}
Let us consider  the real sequence $$u_j:=U(x_j)\qquad j\in\nn$$
and let us show that
$$
\displaystyle \lim_{j\to\infty}u_j = 0 .
$$
 Indeed, the degree $h_j$ of the feedback  $\v$ at every $x_j$ is bounded by ${{k}}$. Moreover,  if we use $r_j$ to denote the positive integer appearing in formula \eqref{yt} in relation with   the feedback $\v$ at $x_j$, we get
 $$r_j^{h_j}\le (r(k))^{k},
 $$
  if we set
  $$
 r(k):=\max\{r_j, \,\,j\in \nn\}  \footnote{This maximum clearly exists and depends (monotonically) on $k$. For instance. $r(2)=4, r(3)=10, r(4) = \max\{22, 16\}=22$.}
.
 $$
 Therefore, by Claim  \ref{claim2h}  we obtain
\bel{diffi}
\begin{array}{l}
u_{j+1}-u_j = U(x_{j+1})-U(x_j)  \le   - \frac{\gamma(u_j)}{2}\,\,\left(\frac{\tau(u_j,h_j)}{r_j}\right)^{h_j}  \\ \, \\
\qquad\qquad\qquad\qquad\qquad\qquad\qquad  \le- \frac{\gamma(u_j)}{2}\,\,\left(\frac{\tau(u_j,{{k}})}{r(k)}\right)^{{{k}}} <0\, ,
  \end{array}
\eeq
for all $j\ge1$.
Hence
the sequence $(u_j)$ is positive  and decreasing, so there exists the limit $$\lim_{j\to\infty} u_j=\eta\ge0.$$
Let us show that  $\eta =0$. If, on the contrary, $\eta$ were
strictly positive, by Claim  \ref{claim2h} one would have $\lim_{j\to\infty} \tau(u_j,{{k}})\ge\tau(\eta,{{k}})>0$.
Hence taking the limit in \eqref{diffi}
one would obtain
$$
0=\eta-\eta\le- \lim_{j\to\infty} \frac{\gamma(u_j)\tau^{{k}}(u_j,{{k}})}{2(r(k))^{{k}}} \le  -\frac{\gamma(\eta)\tau^{{k}}(\eta,{{k}})}{2(r(k))^{{k}}} < 0,
 $$
 a contradiction.
Therefore
$$\lim_{j\to\infty} U(x_j)=\lim_{j\to\infty} u_j = 0.
$$
Hence,  setting  $$S:=\lim_{j\to\infty} s_j\,\,\, =\displaystyle\sum_{i=1}^\infty t_i$$ and
$$
(y,\alpha)(s) := (y_j,\alpha_j)(s)  \quad \forall j\ge1, \,\,\forall s\in [s_{j-1},s_j] ,
$$
one finds that
$$
\lim_{j\to\infty} {\bf d}\left(y(s_j)\right) = 0.
$$
Actually the stronger limit relation
$$
\lim_{s\to S^-} {\bf d}\left(y(s)\right) = 0
$$
holds, as it  follows  from the construction of the  function $\beta$ below.

\subsection{Construction of a bounding  ${\mathcal KL}$ function }
In order to conclude the proof that the system is GAC to $\C$,
we have to establish the existence of  a ${\mathcal KL}$ function $\beta$ such that ${\bf d}(y(s))\le\beta({\bf d}(y(0)),s)$ for every $s\ge0$,
as in   Definition \ref{GAC}.

 By   Claim \ref{claim2h}, for any $t_j=\tau(u_j,h_j)$ one has $t_j^{h_j-1}\ge \tau^{{{k}}-1}(u_j,{{k}})$. Moreover, as already remarked,  $(r_j)^{h_j}\le (r(k))^{k}$ (recall that we are using $r_j$ to denote the positive integer appearing in formula \eqref{yt} in relation with    the feedback $\v$ at $x_j$) .  Hence, for any $j\ge1$,   we have:
$$
\begin{array}{l}
U(y_j(s_j))-U(y_{j}(s_{j-1}))= u_{j+1}-u_{j}    \\ \, \\
\qquad\qquad\qquad\qquad\qquad\le-\frac{\gamma(u_j)}{2}\,\,\left(\frac{t_j}{r_j}\right)^{h_j} \le   -\frac{\gamma(u_j) \tau^{{{k}}-1}(u_j,{{k}})}{2(r({{k}}))^{{{k}}}}\,t_j.
\end{array}
$$
Let us define  the function $\tilde\gamma:(0,\sigma]\to\R$ by setting
\bel{tildeg}
 \tilde\gamma(u):=  \frac{\gamma(u) \,\tau^{{{k}}-1}(u,{{k}})}{2(r({{k}}))^{{{k}}}}.
\eeq
Clearly,  by the monotonicity of $u\mapsto \tau(u,{{k}})$,  $\tilde\gamma$ is (positive and) strictly increasing.
Therefore, since  $U(y(s_j))\le U(y(s_i))$  for every $i=1,\dots,j$,   we get
$$
\begin{array}{l}
U(y(s_j))-U(z)=[U(y(s_j))-U(y(s_{j-1}))]+[U(y(s_{j-1}))-U(y(s_{j-2}))] +\dots  \\ \, \\
 +[U(y(s_{1}))-U(y(0))]\le-\sum_{i=1}^j\tilde\gamma(U(y(s_i)))\,[s_i-s_{i-1}]\le  -\tilde\gamma(U(y(s_j)))\,s_j\,.
 \end{array}
 $$
In particular,  we have
\bel{EP}
U(y(s_j))+\tilde\gamma(U(y(s_j)))\,s_j\le U(z)\,.
\eeq
 We now replace the function $\tilde\gamma$ with    the slightly different function $\hat \gamma:[0,+\infty)\to[0,+\infty)$  defined  by $\hat\gamma(u)\doteq \min\{u,\tilde\gamma(u)\}$   for all $u\in[0,+\infty)$. Notice that $\hat\gamma$   is  continuous,  strictly increasing  and $\hat\gamma(u)>0$ \, $\forall u>0$,  $\hat\gamma(0)=0$. Then,  for any $j\ge 1$,
$$
\hat \gamma(U(y(s_j)))(1+s_j)\le U(z),
$$
so that
\bel{stimau}
 U(y(s_j))\le \hat \gamma^{-1}\left(\frac{U(z)}{1+s_j}\right).
\eeq
 Let $\delta_-$, $\delta_+:[0,+\infty)\to[0,+\infty)$ be the  continuous, strictly increasing, unbounded  functions defined by
\bel{sigma}
\delta_-(u)\doteq\min\{{\bf d}(x): \ U(x)\ge u\}, \quad
 \delta_+(u)\doteq\max\{{\bf d}(x): \ U(x)\le u\}
\eeq
and let us set $\hat\delta_-(u):=\min\{\delta_-(u), u\}$.  Notice that    $\hat\delta_-(0)=\delta^+(0)=0$,  and
$$
\quad\hat\delta_-(U(x))\le {\bf d}(x)\le\delta^+(U(x)),
$$
$\forall x\in{U^{-1}((0,\sigma])}$.
Therefore, setting
\bel{defbeta}
\hat\beta(\delta,s):=\delta_+\circ \hat \gamma^{-1}\left(
\frac{\hat\delta_-^{-1}(\delta)}{1+s} \right) \quad\forall (\delta,s)\in [0,+\infty)\times[0,+\infty) ,
\eeq
by \eqref{stimau} we get
\bel{stimau2}
{\bf d}(y(s_j ))\le \delta_+\left( U(y(s_j )\right)\leq  \delta_+\left(\hat \gamma^{-1}\left(\frac{U(z)}{1+s_j}\right) \right) \le    \hat\beta({\bf d}(z), s_j ),
\eeq
for every  $j\ge 1$. The estimate \eqref{stimau2} says that the function $\hat\beta$ bounds the distance of the trajectory $y(\cdot)$ from the target $\C$ at the discrete times $s_j$.  Hence, in order to  get a bound at
 all times, we need to slightly modify $\hat\beta$. For this purpose, given  any $x\in\R^n\setminus\C$, let us select    a point $\pi(x)\in\C$ such that ${\bf d}(x)=|x-\pi(x)|$.   Notice that for any $s\in[s_j,s_{j+1}]$, one has
 $$
 \begin{array}{l}
 {\bf d}(y(s))\le |y(s)-\pi( y(s_j))|\le |y(s)-y(s_j)|+| y(s_j)-\pi( y(s_j))| \\ \, \\
 \qquad\qquad\qquad\le M [s_{j+1}-s_j]+{\bf d}( y(s_j))  .
 \end{array}
 $$
  Furthermore, by the definition of $\tau$ (see Claim \ref{claim2h}) it follows that
\bel{ed}
 s_{j+1}-s_j=t_{j+1}\le  \tau(u_{j+1},{{k}})\le    \tau(\delta_-^{-1}({\bf d}(y(s_j))),{{k}}).
\eeq
 Therefore,
$$
\begin{array}{l}
 {\bf d}(y(s))\le   M \tau(\delta_-^{-1}({\bf d}(y(s_j))),{{k}})+{\bf d}(y(s_j)) \\ \, \\
\qquad\qquad\qquad \le M \tau(\delta_-^{-1}(\hat\beta({\bf d}(z),s_j)),{{k}})+ \hat\beta({\bf d}(z), s_j )\,.
 \end{array}
$$
Since
for all $\delta$ the function $s\mapsto \hat\beta(\delta,s) $ is decreasing,
one obtains
$$
{\bf d}(y(s ))\le \beta({\bf d}(z), s ) \qquad \forall s\in[0,+\infty[,
$$
where   we have set, for all $s\in\left[0, \tau(\delta_-^{-1}(\delta),{{k}})\right]$,
$$
\beta(\delta,s):= M \tau(\delta_-^{-1}(\hat\beta(\delta,0)),{{k}})+ \hat\beta(\delta, 0 )
$$
and, if $s>\tau(\delta_-^{-1}(\delta),{{k}})$,
$$
\beta(\delta,s):= M \tau(\delta_-^{-1}(\hat\beta(\delta,s-\tau(\delta_-^{-1}(\delta),{{k}}))),{{k}})+ \hat\beta(\delta, s-\tau(\delta_-^{-1}(\delta),{{k}}) ).
$$
 \subsection{Removal of the fictitious $C^{k-1}$-bound}
 Let us see now that, by means of a cut-off argument, we can  remove the auxiliary boundedness hypothesis \eqref{fieldbounded}.
Let   $\psi:\R^n\to[0,1]$ be a $C^{\infty}$ map such that
\bel{psi}
 \psi = 1 \quad\hbox{on}\quad \ol{B(\C,\nu)\setminus \C}, \qquad \psi = 0
 \quad\hbox{on}\,\,\,  \rr^n\backslash B(\C,2\nu)
\eeq
and consider the control system
\bel{cutode}
\xi'=\sum_{i=1}^ma_i\,\left(\psi f_i(\xi)\right)\,
\eeq
 Notice that the functions $(\psi\,f_i)$ belong to $C_b^{k-1}(\R^n\setminus\C)$  because of the cut-off factor $\psi$.
 Moreover,  any trajectory $\xi(\cdot)$ of \eqref{cutode} with the initial condition $z\in  U^{-1}((0,\sigma])$,   exists globally and cannot exit the  compact  set $\ol{B(\C,2\nu)\setminus\C}$.  Owing to  the previous step,  there exists a trajectory $\xi$ which  approaches asymptotically the target and verifies ${\bf d}(\xi(s ))\le \beta({\bf d}(z), s ) \qquad \forall s\in[0,+\infty[$. Moreover,  $\xi(s)$ belongs to $B(\C,\nu)$ for every $s\ge0$. Therefore   $\xi$ is a solution of the original system, proving that  \eqref{odeL} has the GAC property in $U^{-1}((0,\sigma])$.

 By the arbitrariness of $\sigma>0$, it is easy to extend these constructions  from $U^{-1}((0,\sigma])$   to the whole set $\R^n\setminus\C$.
   This concludes the proof of \eqref{Th3.1gen}. \qed


\section{The case of  nonsmooth dynamics}\label{nonsmooth}

 Let us begin with an example, where the vector fields are not $C^1$.

\begin{Example}\label{esL}  {\rm Let us consider the system
\bel{nhmod}
\dot y =a_1\,f_1(y)  +a_2\,  f_2(y)\,,
\eeq
where
$$
f_1:= \frac{\partial}{\partial x_1}
+ (|x_2|-2x_2)\frac{\partial}{\partial x_3}\,, \quad f_2:= \frac{\partial}{\partial x_2}+(|x_1|+2x_1)\frac{\partial}{\partial x_3},
$$
and let the target $\C$ coincide with the origin.  For the same reason as in the nonholonomic integrator, a smooth degree-$1$ \cf\, does not exist (see Example \ref{es1}).  On the other hand, the given  notion of degree-$2$  control Lyapunov function is not even  meaningful here,
in that  the classical bracket $[f_1,f_2]$ is not   defined at points $x$ such that $x_1=0$ or $x_2=0$.  Yet, in the open, dense set $\{x: \ x_1\ne 0, \ x_2\ne0\}$  the bracket is well defined and, furthermore,   the Lie algebra rank condition is verified. So it is reasonable to look for a Lyapunov-like condition involving {\it somehow } Lie brackets.
}
\end{Example}

 One one hand,   in  the definition of degree-$k$ \cf \, one requires the vector fields $f_1,\ldots,f_m$ to be  of class $C^{k-1}$ (see also Remark \ref{wr}), for this   guarantees  that  the Lie brackets up to the degree $k$ are well defined and continuous.
 On the other hand, as remarked in the Introduction, the non smoothness of  a degree-$1$ \cf\,   is more related to a shortage of  dynamics' directions than to the regularity of the involved vector fields.

Let us introduce the notion of set-valued bracket for locally Lipschitz vector fields.
\begin{Definition}[\cite{RS1}]\label{Lb} Let $\Omega\subset\R^n$ be an open set and let $f$, $g$ be   vector fields belonging to $C^{0,1}(\Omega)$. For every $x\in\Omega$, let us set
$$\begin{array}{c}
 [f,g]_{set}(x):= \\\,\\co\Big\{v\in\rr^n , \quad  v=\lim_{x_n\to x}  [f,g](x_n) \quad (x_n)_{n\in\nn}\subset DIFF(f)\cap  DIFF(g)\Big\},
\end{array}
$$
where $co$ means "convex hull", and $DIFF(f)$, $DIFF(g)$ denote  the   subsets of differentiability points of $f$ and $g$, respectively. Let us observe that $DIFF(f)$ and $DIFF(g)$ have full measure, hence they are dense in $\Omega$.
\end{Definition}
Notice that, as in the regular case, one has $[f,f]_{set}(x)=\{0\}$ and  \linebreak
$[f,g]_{set}(x)=- [g,f]_{set}(x)$ for every $x\in\Omega$ (for any $E\subset\R^n$ we use the notation $-E=\{-v: \ v\in E\}$).

\vv
Let us consider the control system
 \bel{odeLlip}
\left\{ \begin{array}{l}
\ds\dot y  =\sum_{i=1,\ldots,m} a_i\,f_i(y)\\ \, \\
y(0)=x\in \R^n\setminus\C
\end{array}\right.
\eeq
and let us assume that  $f_1,\dots,f_m$ belong to $C_b^{0,1}(\Omega\setminus\C)$  for any bounded, open set $\Omega\subset\R^n$ (see Subsection \ref{prel}).

\vv
The families  $\F^{(1)}$ and   $\F^{(2)}$ are formally defined   as in the regular (i.e., $C^1$) case, except that  their elements  are set-valued vector fields \footnote{$\F^{(1)}$ practically coincides with the one in the regular case, for its elements are set-valued maps which are singletons, indeed.}:
$$
\F^{(1)}:=  \Big\{\{f_\ell(\cdot)\}\, , \, \{-f_\ell(\cdot)\} : \quad \ell=1,\dots,m\Big\},
$$
and
$$
\F^{(2)}:= \F^{(1)}\cup  \Big\{ [f_i,f_j]_{set}(\cdot) : \quad i,j=1,\dots,m\Big\}.
 $$
 As in the regular case,  for every $x\in\R^n\setminus\C$, we set
$$
\F^{(1)}(x):=  \Big\{\{f_\ell(x)\}\, , \, \{-f_\ell(x)\} : \quad \ell=1,\dots,m\Big\},
$$
and
$$
\F^{(2)}(x):= \F^{(1)}(x)\cup  \Big\{ [f_i,f_j]_{set}(x) : \quad i,j=1,\dots,m\Big\}.
 $$
 Accordingly, for  $h=1,2$, we define the {\it degree-$h$ Hamiltonian}  $H^{({{h}})}$ by  setting
$$
H^{({{h}})}(x, p):= \inf_{v\in {\F^{({{h}})}(x)}}\,\,\sup_{w\in v}\Big\langle p ,\, w\Big\rangle,
$$
 for all $(x,p)\in (\R^n\setminus\C)\times \rr^n$.
More explicitly, one has
$$
 H^{(1)}(x,p)  =\displaystyle  \inf\left\{-\left|\big\langle p ,\, f_\ell(x)\big\rangle\right|: \quad \ell=1,\dots,m\right\},
 $$
and
$$
H^{({{2}})}(x, p)=\inf_{\ell, i, j}\left\{-\left|\big\langle p ,\, f_\ell(x)\big\rangle\right|\, , \, \sup_{w\in  [f_i,f_j]_{set}(x)}\big\langle p ,\, w \big\rangle : \quad \ell, i,j=1,\dots,m\right\}.
$$

 We can now state, for the case $k=2$,  a  generalization of  Theorem  \ref{Th3.1gen}  to the nonsmooth case:
\begin{Theorem}\label{Th3.1genL}
Let us assume that   a degree-$2$ control Lyapunov function  exists. Then   system {\rm (\ref{odeLlip})}  is {\rm GAC} to $\C$.
\end{Theorem}

\addtocounter{Example}{-1}
\begin{Example}[continued]{\rm   Let us come  back to the system in  \eqref{nhmod} and compute the bracket $[f_1,f_2]_{set}$. It turns out that
$$
[f_1,f_2]_{set}(x)=I(x)\frac{\partial}{\partial x_3},
$$
where
$$
I(x):=\left\{\begin{array}{l}
 \{4\} \qquad\text{if $x_1x_2>0$}, \\
\{2\}   \qquad\text{if $x_1<0$ and $x_2>0$}, \\
\{6\}  \qquad\text{if $x_1>0$ and $x_2<0$}, \\
\left[2,4\right]  \quad\text{if either $x_1=0$ and $x_2>0$, or  $x_1<0$ and $x_2=0$}, \\
\left[4,6\right] \quad\text{if either $x_1=0$ and $x_2<0$, or  $x_1>0$ and $x_2=0$}, \\
\left[2,6\right] \quad\text{if $x_1=x_2=0$.}
\end{array}\right.
$$
The distance function $U(x)=|x|$ is not a degree-$1$  \cf. Indeed, by
$$
H^{(1)}(x,p)=\inf\Big\{-|p_1+p_3(|x_2|-2x_2)|\, , \, -|p_2+p_3(|x_1|+2x_1)|\Big\}
$$
one obtains
$$
H^{(1)}\big((0,0,x_3),DU(0,0,x_3)\big)=0 \quad \text{for every $x_3\ne0$.}
$$
Yet, the distance function   $U$ happens to be  a ($C^\infty$)  degree-$2$  \cf.  Indeed,
$$
\begin{array}{l}
H^{(2)}(x,p)=
\inf\Big\{H^{(1)}(x,p) \, , \, \sup_{w\in  I(x)}\,w\,p_3\, , \, \sup_{w\in - I(x)}\,w\,p_3\Big\}= \\ \, \\
-\sup\Big\{|p_1+p_3(|x_2|-2x_2)|\, , \, |p_2+p_3(|x_1|+2x_1)|\, , \, 2|p_3|\Big\}
\end{array}
$$
and, since  $|DU(x)|=1$ for every $x\ne0$, one gets
\bel{ff}
H^{(2)}\big(x,DU(x)\big)<0.
\eeq
Notice that, for the validity of the strict inequality in \eqref{ff}, it is crucial that $0\notin [f_1,f_2]_{set}(x)$ for every $x\ne0$. Furthermore, arguing as in Remark \ref{InvL}, we know that a (possibly nonsmooth)  degree-$1$ \cf \, does exist. Actually, the function
$$
U(x)=\max \Big\{\sqrt{x_1^2+x_2^2}, |x_3|-\sqrt{x_1^2+x_2^2}\Big\}
$$
introduced in \cite{Ri} as a control Lyapunov function for the nonholonomic integrator is a degree-$1$ \cf \,also for this system.
}
\end{Example}

\subsection{Proof of Theorem    \ref{Th3.1genL}}
The proof is akin to the proof of  Theorem  \ref{Th3.1gen}. Yet,  because of the new kind of brackets and Hamiltonians here involved, some changes are needed.

 As in the regular case,  the $0$ in   the dissipative  relation   can be replaced by a nonnegative function of $U$:
\begin{Proposition}\label{claim1Lip}  Let  $U:\overline{\R^n\setminus\C} \to\R$   be a continuous function, such that $U$ is   locally  semiconcave, positive definite and proper on  ${\R^n\setminus\C}$.  Then the conditions (i) and (ii) below are equivalent:

\begin{itemize}\item[(i)]  $U$ verifies
$$
H^{({{2}})}(x, D^*U(x))<0 \quad\text{for all $x\in\R^n\setminus\C$;}
$$
\item[(ii)] for every $\sigma>0$  there exists a continuous, strictly increasing, function  $\gamma:[0,+\infty)\to:[0,+\infty)$ such that
$$
H^{({{2}})}(x,D^*U(x)) \le-\gamma(U(x)) \quad\text{for all $x\in U^{-1}((0,2\sigma])$.}
$$
\end{itemize}
\end{Proposition}
\begin{proof} By \cite{RS1},  for every $i,j$, the set-valued map $x\mapsto [f_i,g_j]_{set}(x)$ is upper semicontinuous, with compact, convex values. Moreover,   basic results on marginal functions imply that  $(x,p)\mapsto \sup_{w\in  [f_i,g_j]_{set}(x)}\big\langle p ,\, w \big\rangle$  is  upper semicontinuous  (see e.g. \cite{AC}). As an easy consequence,  the Hamiltonian
$$
(x,p)\mapsto H^{(2)}(x, p)
$$
 turns out to be  upper semicontinuous. At this point one can conclude,  arguing  exactly as in Proposition \ref{claim1}.
\end{proof}

We now need to adapt the notion of {   degree-$2$  feedback} to  the case when Lie brackets are set-valued.    For a given $\sigma>0$, let  ${{\gamma}}$  be a  function as in
Proposition \ref{claim1Lip}, and let $x\mapsto p(x)$ be a selection of  $x\mapsto D^*U(x)$ on $U^{-1}((0,2\sigma])$, so that
$$
H^{({{2}})}(x,p(x)) \le-\gamma(U(x)) \qquad \forall x\in U^{-1}((0,2\sigma]).
$$

 \begin{Definition}\label{DfedLip} For a given  $\sigma>0$, let  ${{\gamma}}(\cdot)$, and $p(\cdot)$ be chosen as above.
A selection  $$ \v :  U^{-1}((0,2\sigma]) \to 2^{\rr^n} , \qquad
 x  \mapsto \v(x)\in \F^{({{2}})}(x)$$  is called  a {\em  degree-$2$ feedback}   (corresponding to $U$, $\sigma$, ${{\gamma}}$, and $p(\cdot)$)  if for every $x\in U^{-1}((0,2\sigma])$  there exists  $h\in\{1,2\}$ such that
\bel{feed1lip}
\left\{\begin{array}{l}
\v(x)\in\F^{(h)}(x), \\\,\\
\ds \sup_{w\in \v(x)}\Big\langle p(x) ,\, w  \Big\rangle \le -{{\gamma}}(U(x)),   \\ \,\\
 \text{and, if $h=2$,} \quad H^{({{1}})}(x,p(x))>-\gamma(U(x)).
\end{array}\right.
\eeq
The number  $h$ will be called the {\em degree of the feedback  $\v$ at  $x$}.
\end{Definition}
More explicitly, when $h=1$,  for some $\ell=1,\dots,m$  one has
$$
\left\{\begin{array}{l}
\v(x)=\{f_\ell(x)\} \ \text{or} \   \v(x)=-\{f_\ell(x)\} \\\,\\
 -\left|\Big\langle p(x) ,\, f_\ell(x)  \Big\rangle\right| \le -{{\gamma}}(U(x));
 \end{array}\right.
$$
if $h=2$,  for some $i,j=1,\dots,m$ ($i\ne j$)  one has
$$
\left\{\begin{array}{l}
\v(x) =[f_i,f_j]_{set}(x), \\\,\\
 \Big\langle p(x) ,\, w  \Big\rangle \le -{{\gamma}}(U(x)) \quad \forall w\in [f_i,f_j]_{set}(x),   \\ \,\\
 \text{and} \quad H^{({{1}})}(x,p(x))>-\gamma(U(x)).
\end{array}\right.
$$

\vskip0.3truecm

The following  claim is  a version  of  Claim  \ref{claim2h}  adapted to the set-valued notion of feedback:
 \begin{Claim}\label{claim2hL}  Let $U(\cdot)$,  $\sigma$, $ \gamma(\cdot)$ and  $p(\cdot)$ as above. Furthermore let $\v(\cdot)$ be a (set-valued) degree-$2$ feedback corresponding to these data. Then there exists a time-valued  function
$$
\tau:(0,\sigma]\times \{1,2\}\to (0,1]\,,
$$
 such that
\begin{itemize}
\item[i)] $j\mapsto \tau(u,j)^j$ and  $j\mapsto\tau(u,j)^{j-1}$ are decreasing  for every $u\in(0,\sigma]$,
\item[ii)]  $u\mapsto \tau(u,j)$ is increasing  for every $j\in\{1,2\}$,  and
\item[iii)] for all  $z\in U^{-1}((0,\sigma])$ with a feedback    $\v(\cdot)$ of degree $h$ at $z$,  one has
 $$
U(y_z(t,\alpha_t))-U(z)\le  - \frac{\gamma(U(z ))}{2}  \left(\frac{t}{r}\right)^{h}\,\quad \forall t\in\left[0,\,\tau(U(z),h)\right]\,,
$$
where $r$ and   $y_z(\cdot,\alpha_t)$  are an integer and a  trajectory  associated to  $\v(z)$ according to Lemma \ref{applemmaL}.
\end{itemize}
\end{Claim}

The proof of Claim  \ref{claim2h} was based the  on the  asymptotic formulas stated in  Lemma \ref{applemma}. Similarly,  Claim \ref{claim2hL}  results  proved as soon as one applies the  asymptotic formula stated in Lemma \ref{applemmaL}  below. Precisely, once arrived to formula \eqref{tau}, one simply replaces   estimate \eqref{yt} with \eqref{ytL1},  and observes  that,  by the definition of $H^{(2)}$,    the inequality
  $\Big\langle p(z) ,\, w  \Big\rangle \le -{{\gamma}}(U(z))$ is   verified {\it for all}  $w\in [f_i,f_j]_{set}(z)$ (see Definition \ref{DfedLip}).

 \begin{Lemma}[\cite{RS2}, \cite{FR2}]\ \label{applemmaL} If   $f_1,\dots,f_m\in C_b^{0,1}(\R^n\setminus\C)$,  there exists a constant  $c>0$ such that for any $x\in \R^n\setminus\C$, any feedback    $\v(x)= [f_i,f_j]_{set}(x)$,  and any $t>0$,
setting
$$
\alpha_t(s)=e_i\chi_{[0,t/4[}(s)+e_j\chi_{[t/4,t/2[}(s)-e_i\chi_{[t/2,3t/4[}(s)-e_j\chi_{[3t/4,t]}(s), \quad s\in[0,t],
$$
 the estimate
\bel{ytL}
d\left( y_x(t,\alpha_t)-x \, , \,  [f_i,f_j]_{set}(x)\, \frac{t^2}{16}  \right)\leq\frac{c}{16}\, t^3
\eeq
holds true. In particular, for any $t>0$ there  exists some $w(t)\in \v(x)$ such that
\bel{ytL1}
\left|y_x(t,\alpha_t)-x-w(t)\frac{t^2}{16} \right| \leq\frac{c}{16}\, t^3.
\eeq
\end{Lemma}

 Let us point out that Claim  \ref{claim2h} is  the starting point for the construction of an admissible trajectory-control pair by means of  the recursive procedure described  in the proof of  Theorem  \ref{Th3.1gen}. Through exactly  the same arguments and  in view of Claim  \ref{claim2hL}, the  proof of Theorem  \ref{Th3.1genL} can now  be completed.


\section{Concluding remarks}\label{Snew}

\subsection{Feedback constructions} Degree-1 control Lyapunov functions are  used as primary ingredient for the construction  of feedback stabilizing  strategies, a classical question that is mainly concerned with the definition of an appropriate notion of solution for discontinuous ODE's (see e.g. \cite{CLSS}, \cite{CLRS}, \cite{MaRS} and \cite{AB}).

It  might be interesting  to associate  some  concept
 of feedback strategy also  in relation with  a degree-$k$  control Lyapunov function, $k>1$, all the more so as it may happen  to be smoother than the degree-1 CLF. As a matter of fact, the proofs of Theorems \ref{Th3.1gen} and \ref{Th3.1genL} seem to suggest a notion of
 "feedback" such that, in dependence of the bracket that   minimizes the Hamiltonian $H^{(2)}$,  singles out a suitable finite sequence  of constant controls  to be implemented along small time intervals.

\subsection{ Degree-$k$ CLF as  viscosity supersolutions}
 One is obviously tempted to refer to a degree-$k$ control Lyapunov function  as a
 {\it strict supersolution} of the Hamilton-Jacobi equation
$$
  - H^{({{k}})}(x,DU(x)) =0.
  $$
  Actually this  holds true as soon as  we consider, for instance,  the notion of  viscosity solution. More precisely, one has:

{\it  Let $U:\overline{\R^n\setminus\C}\to\R$ be  a  continuous function which, furthermore, is  locally  semiconcave, positive, and proper on ${\R^n\setminus\C}$. Then $U$  is a  degree-${{k}}$ control Lyapunov function if and only for any $N>0$ there is some continuous, strictly increasing function $\gamma:[0,+\infty)\to[0,+\infty)$ such that $U$ is a viscosity supersolution   \footnote{Namely, $- H^{({{k}})}(x,p)\ge \gamma(U(x))$ for all $x\in {\R^n\setminus\C}$ and $p$ in the subgradient $DU^-(x)$ of $U$ at $x$ (see \cite{BCD}). } of
\bel{visco1}  - H^{({{k}})}(x,DU(x)) =\gamma(U(x)) \quad\text{in $U^{-1}((0,N))$.}
\eeq
    }
 Indeed, in the case of a locally  semiconcave function $U$,  at every $x\in DIFF(U)$ the subdifferential  $D^-U(x)=\{\nabla U(x)\}$ coincides with $ D^*U(x)$,  while $D^-U(x)$ is  empty  if $x\notin D^-U(x)$. Therefore, thanks to Proposition  \ref{claim1}, \eqref{visco1} follows from   the inequality
\bel{hkdim}
H^{(k)}(x,D^*U(x))<0,
\eeq
which defines the notion of degree-$k$ \cf.  To obtain the converse implication, for any $x\in U^{-1}((0,N))$ and $p\in D^*U(x)$,  let $(x_n)_n\subset U^{-1}((0,N))\cap\, DIFF(U)$ be such that $\lim_n(x_n, \nabla U(x_n))=(x,p)$.  Then, by  hypothesis (\ref{visco1}),  one has
$$
-H^{(k)}\left(x_n, \nabla U(x_n)\right)\ge  \gamma(U(x_n)) \qquad\forall n\in\N,
$$
so,  passing to the limit, one gets \eqref{hkdim}.

\subsection{ Generalizations to larger classes of systems}\label{ssgen}
As it is well known, the lack of symmetry poses non trivial problems for controllability. The same kind of difficulty is therefore encountered in the attempt to define a reasonable notion of   degree-$k$ \cf\,
 for systems
\bel{Ed}
 \left\{\begin{array}{l} \dot y =f_0(y) +\ds \sum_{i=1,\ldots,m}a_if_i(y) \\ \, \\
  y(0)=x\in\rr^n\backslash\C
  \end{array}\right.
\quad a\in \{0, \pm e_1,\ldots,e_m\}
\eeq
having a  {\it non zero drift} $f_0$. Let us assume that $f_0,\dots,f_m$ belong to $C^1_b(\Omega\setminus\C)$ for any open, bounded set $\Omega\subset\R^n$.

  Let us examine    the case $k=2$.  Intuition coming from controllability literature suggests that a notion of  degree-$2$ control Lyapunov function should   be shaped  in such  a way that it  would be allowed to  violate the standard dissipative  inequality  only at the points where the drift $f_0$ vanishes.
Accordingly,   let us redefine the classes of vector fields
$$
\F^{(1)}:=  \Big\{f_0,  f_0 - f_i, f_0 +f_i\quad i=1\ldots,m\Big\},
$$
    $$
\F^{(2)}:=\F^{(1)}\cup  \Big\{\pm [f_0,f_i]\cdot\chi_{\{f_0=0\}},  [f_j,f_\ell]\cdot\chi_{\{f_0=0\}}\quad i,j,\ell=1\ldots,m\Big\},
$$
and the Hamiltonians
$$
H^{({{h}})}(x, p):= \inf_{g\in {\F^{({{h}})}(x)}}\big\langle p ,\, g\big\rangle \quad h=1,2.
$$
 Accordingly, one might call  {\em degree-${{2}}$ control Lyapunov function  }
any  continuous  function $U:\overline{\rr^n\backslash\C}\to\R$ such that its restriction  to $\rr^n\backslash\C$ is  locally  semiconcave, positive definite, proper, and verifies
$$H^{({{2}})}(x,D^*U(x))  < 0 \quad \forall x \in \rr^n\backslash\C .$$
In particular, $U$ is a   degree-${{2}}$ \cf\, if, at each point $x\in\rr^n\backslash\C$, \begin{small}either
$$
\min \Big\{ \langle D^*U(x),f_0(x)\rangle,\langle D^*U(x),(f_0 -f_i)(x)\rangle,\langle D^*U(x),(f_0 +f_j)(x)\rangle, \ i,j= 1,\ldots,m\Big\} <0,
$$
or
$$\left\{\begin{array}{l}
f_0(x)=0, \\\,\\
\min \Big\{ -|\langle D^*U(x),[f_0,f_i](x)\rangle|, -|\langle D^*U(x),[f_j,f_\ell(x)\rangle|\quad i,j,\ell= 1,\ldots,m  \Big\} <0.\end{array}\right.$$\end{small}
With these settings, we get  the following result:

\begin{Theorem}\label{Th51}
{\it  Let a degree-${{2}}$ control Lyapunov function exist.
Then   system {\rm (\ref{Ed})} is GAC to $\C$.
}
\end{Theorem}

 A proof  of this result can be deduced by  first observing that  Lemma \ref{applemma} has a counterpart in  an asymptotic formula  for brackets of the form $[f_0,f_i]$, $[f_j,f_\ell]$ valid at all  points $x$ where $f_0(x)=0$.
 Precisely, through standard Taylor expansions one can prove the following result:
 \begin{Lemma} \label{applemmaD} If $f_0,f_1,\ldots,f_m\in C_b^1(\R^n\setminus\C)$,  there exists a constant  $c>0$ such that for any $x\in \R^n\setminus\C$ where $f_0(x) =0$, any  $j=1,\ldots,m$, and any $t>0$,
 the following estimates hold true:
\bel{ytL1}
d\left( y_x(t,\alpha_t)-x \, , \,  [f_0,f_i](x)\, \frac{t^2}{4}  \right)\leq\frac{c}{4}\, t^3\quad \forall s\in [0,t]
\eeq
with
$$
\alpha_t(s):=e_i\chi_{[0,t/2[}(s)-e_i\chi_{[t/2,t[}(s) ;
$$
\bel{ytL2}
d\left( y_x(t,\hat\alpha_t)-x \, , \,  [f_i,f_0](x)\, \frac{t^2}{4}  \right)\leq\frac{c}{4}\, t^3, \forall s\in [0,t]
\eeq
with
$$
\hat \alpha_t(s):=-e_i\chi_{[0,t/2[}(s)+e_i\chi_{[t/2,t[}(s);
$$
and
\bel{ytL3}
d\left( y_x(t,\tilde\alpha_t)-x \, , \,  [f_j,f_\ell](x)\, \frac{t^2}{16}  \right)\leq\frac{c}{16}\, t^3, \quad s\in[0,t]
\eeq
with
$$
\tilde\alpha_t(s)=e_j\chi_{[0,t/4[}(s)+e_\ell\chi_{[t/4,t/2[}(s)-e_j\chi_{[t/2,3t/4[}(s)-e_\ell\chi_{[3t/4,t]}(s).
$$

\end{Lemma}
Hence, provided one gives  a (obvious) notion of degree-2 feedback at the points $x$ where
$f_0(x)=0$, the  proof of Theorem \ref{Th51} can be easily  achieved by means of   the same arguments as in the proof of Theorem \ref{Th3.1gen}.

\begin{Example}{\rm Consider the so-called {\it soft landing} problem, $\dot y_1=y_2$ , $\dot y_2 = a$, $a\in\{0,-1,1\}$, $\C=\{(0,0)\}$.  The distance function $U(x)=|x|$ --as well as $|x|^{\alpha}$, \,  $\alpha>1$-- fails to be a degree-$1$ \cf \,(because the required inequality is not strict  on the $x_1$ axis). However, $U(x)$ is a  degree-$2$ \cf. } \end{Example}
\vskip0.3truecm
 A reasonable and useful notion  of degree-$k$ \cf \, can be likely  given for the general  case $k\geq 1$    provided only  suitable subsets of brackets are
 allowed\footnote{For instance, one can consider  "good" brackets (see e.g. \cite{Co}), of which $[f_0,f_i]$, $[f_j,f_\ell]$ are degree $2$  instances. Among  degree-$3$ brackets, $[f_0,[f_0,f_i]]$ is good for every $i=1,\ldots,m$, while $[f_0,[f_j,f_\ell]]$ is not good, for all $j,\ell=1,\ldots,m$}.\newline

 \end{document}